\documentclass[12pt]{article}

\usepackage{float}
\usepackage{graphicx}
\usepackage{nameref}
\usepackage[shortlabels,inline]{enumitem}
\usepackage{mathtools}
\usepackage{empheq}
\usepackage[shortlabels]{enumitem}
\setlist[enumerate]{nosep}
\usepackage[doc]{optional}
\usepackage{xcolor}
\usepackage[colorlinks=true,
            linkcolor=refkey,
            urlcolor=lblue,
            citecolor=red]{hyperref}
\usepackage{float}
\usepackage{soul}
\usepackage{graphicx}

\definecolor{labelkey}{rgb}{0,0.08,0.45}
\definecolor{refkey}{rgb}{0,0.6,0.0}
\definecolor{Brown}{rgb}{0.45,0.0,0.05}
\definecolor{lime}{rgb}{0.00,0.8,0.0}
\definecolor{lblue}{rgb}{0.5,0.5,0.99}

 \usepackage{mathpazo}

\colorlet{hlcyan}{cyan!30}

\colorlet{hlred}{red!40}

\usepackage{stmaryrd}
\usepackage{amssymb}
\makeatletter
\def\namedlabel#1#2{\begingroup
   \def\@currentlabel{#2}%
   \label{#1}\endgroup
}
\newcommand{\vast}{\bBigg@{4}}
\newcommand{\Vast}{\bBigg@{5}}
\makeatother

\providecommand{\siff}{\Leftrightarrow}
\newcommand{\moyo}[2]{\ensuremath{\sideset{^{#2}}{}{\operatorname{}}\!#1}}

\newcommand{\nnn}{\ensuremath{{n\in{\mathbb N}}}}

\newcommand{\thalb}{\ensuremath{\tfrac{1}{2}}}

\newcommand{\fenv}[1]%
{\ensuremath{\,\overrightarrow{\operatorname{env}}_{#1}}}
\newcommand{\benv}[1]%
{\ensuremath{\,\overleftarrow{\operatorname{env}}_{#1}}}

\newcommand{\scal}[2]{\left\langle{#1},{#2}  \right\rangle}

\newcommand{\RR}{\ensuremath{\mathbb R}}

\newcommand{\RPP}{\ensuremath{\mathbb{R}_{++}}}

\newcommand{\dom}{\ensuremath{\operatorname{dom}}}

\newcommand{\prox}{\ensuremath{\operatorname{Prox}}}
\newcommand{\ran}{\ensuremath{{\operatorname{ran}}\,}}

\newcommand{\cdom}{\ensuremath{\overline{\operatorname{dom}}\,}}
\newcommand{\intdom}{\ensuremath{{\operatorname{int\,dom}}\,}}

\newcommand{\Id}{\ensuremath{\operatorname{Id}}}

\newcommand{\pinf}{\ensuremath{+\infty}}

{\begin{list}{}{%
\settowidth{\labelwidth}{\textrm{#1~}}%
\setlength{\leftmargin}{\labelwidth+\labelsep}}}%requires macro calc.sty
{\end{list}}
\usepackage{amsthm}
\usepackage[capitalize,nameinlink]{cleveref}
\crefname{figure}{Figure}{Figures}
\crefname{equation}{}{equations}
\crefname{chapter}{Appendix}{chapters}
\crefname{item}{}{items}
\crefname{enumi}{}{}

\theoremstyle{definition}
\newtheorem{theorem}{Theorem}[section]
\newtheorem{lemma}[theorem]{Lemma}

\newtheorem{corollary}[theorem]{Corollary}

\newtheorem{proposition}[theorem]{Proposition}

\newtheorem{definition}[theorem]{Definition}

\newtheorem{example}[theorem]{Example}

\newtheorem{fact}[theorem]{Fact}
\newtheorem{remark}[theorem]{Remark}

\DeclarePairedDelimiter{\parens}{\lparen}{\rparen}

\providecommand{\RR}{\mathbb{R}}

\providecommand{\ran}{\operatorname{ran}}

\providecommand{\dom}{\operatorname{dom}}

\providecommand{\gr}{\operatorname{gra}}

\providecommand{\Id}{\operatorname{{ Id}}}

\providecommand{\gr}{\operatorname{gra}}

\providecommand{\ran}{\operatorname{ran}}

\providecommand{\Id}{\operatorname{Id}}

\newcommand{\cran}{\ensuremath{\overline{\operatorname{ran}}\,}}

\providecommand{\RR}{\mathbb{R}}

\definecolor{myblue}{rgb}{.8, .8, 1}

\allowdisplaybreaks % or locally if problems {\allowdisplaybreaks
\begin{document}

\title{\textsf{
On Carlier's inequality
}}

\author{
Heinz H.\ Bauschke\thanks{
Department of Mathematics, University
of British Columbia,
Kelowna, B.C.\ V1V~1V7, Canada. E-mail:
\texttt{heinz.bauschke@ubc.ca}.},~
Shambhavi Singh\thanks{
Department of Mathematics, University
of British Columbia,
Kelowna, B.C.\ V1V~1V7, Canada. E-mail:
\texttt{sambha@student.ubc.ca}.},~ 
and
Xianfu Wang\thanks{
Department of Mathematics, University
of British Columbia,
Kelowna, B.C.\ V1V~1V7, Canada. E-mail:
\texttt{shawn.wang@ubc.ca}.}
}

\date{June 29, 2022} 
\maketitle

\vskip 8mm

\begin{abstract} 
The Fenchel-Young inequality is fundamental in Convex Analysis and Optimization. It states that the difference between certain function values of two vectors and their inner product is nonnegative.
Recently, Carlier introduced a very nice sharpening of this inequality, 
providing a lower bound that depends on a positive parameter. 

In this note, we expand on Carlier's inequality in three ways. 
First, a duality statement is provided. Secondly, we discuss asymptotic behaviour as the underlying parameter approaches zero or infinity. 
Thirdly, relying on cyclic monotonicity and 
associated Fitzpatrick functions, we present a lower bound that features an infinite series of squares of norms. Several examples illustrate our results.
\end{abstract}

{\small


\section{Introduction}

Throughout the paper, we assume that 
\begin{equation}
\text{$X$ is a real Hilbert space}
\end{equation}
with inner product $\scal{\cdot}{\cdot}$ and induced norm $\|\cdot\|$. 
We also assume throughout that 
\begin{equation}
\label{e:f}
f\colon X\to\RR \;\;\text{is convex, lower semicontinuous, and proper,}
\end{equation}
and that 
\begin{equation}
\label{e:A}
A\colon X\rightrightarrows X \;\;\text{is a maximally monotone operator on $X$.}
\end{equation}
Recall that the \emph{Fenchel conjugate} $f^*$ of $f$ is defined by 
$f^*(x^*)=\sup_{x\in X}(\scal{x}{x^*}-f(x))$. 
The classical Fenchel-Young inequality states that 
for $x$ and $x^*$ in $X$, we have 
\begin{equation}
\label{e:FY}
G(x,x^*) := G_f(x,x^*) := f(x)+f^*(x^*) - \scal{x}{x^*}  \geq 0,
\end{equation}
and we have the well known equality characterization
\begin{equation}
\label{e:FY=}
G(x,x^*)=0
\;\;\Leftrightarrow\;\;
x^*\in\partial f(x). 
\end{equation}
(We assume the reader has some basic knowledge of convex analysis and monotone
operator theory as can be found, e.g., in \cite{BC2017}, \cite{MN}, \cite{Rocky}, 
and \cite{RW}.)
In \cite{Carlier}, Carlier proved recently the following stunningly 
beautiful sharpening 
of \cref{e:FY}: 
\begin{equation}
\label{e:220515a}
f(x)+f^*(x^*) -\scal{x}{x^*}\geq \frac{\|x-\prox_{\gamma f}(x+\gamma x^*)\|^2}{\gamma}, 
\end{equation}
where $\gamma>0$ and $\prox_{\gamma f}$ is the proximal mapping of $\gamma f$. 
He also discusses applications and connections to optimal transport, 
the Brondsted--Rockafellar theorem, and tilted duality.

\emph{The aim of this paper is to expand on Carlier's work in three ways:
(1)~duality (see \cref{t:duality}), (2)~asymptotic behaviour (see \cref{t:asy} 
and \cref{c:asy}), and (3)~cyclic monotonicity (\cref{c:superC}).}

The remainder of this paper is organized as follows. 
In \cref{sec:F_A}, inequalities are provided based on 
the Fitzpatrick function. 
Useful identities involving the Minty parametrization are presented in 
\cref{sec:Minty}. 
Carlier's inequality and a new duality result are given in \cref{sec:Carlier}. 
In \cref{sec:asymp}, we discuss the behaviour of the right side 
of \cref{e:220515a} when $\gamma\to 0^+$ and $\gamma\to\pinf$. 
Various examples are presented in \cref{sec:examples} to illustrate our
results. In \cref{sec:ncyc}, we obtain sharpenings when the underlying 
operator $A$ is cyclically monotone of an order bigger than $2$. 

The notation employed in this paper is fairly standard and follows largely \cite{BC2017}.

\section{The Fitzpatrick function} 
\label{sec:F_A}

In this section, we start the approach to Carlier's result.  
Several of the proofs are implicit in Carlier's work; however, we
include for completeness and the reader's convenience. 
Recall that the \emph{Fitzpatrick function} for the operator $A$ 
(see \cref{e:A}) at $(x,x^*)\in X\times X$ is given by 
\begin{subequations}
\label{F_A}
\begin{align}
F_A(x,x^*) &:= 
\scal{x}{x^*} - \inf_{(a,a^*)\in\gr A}\scal{x-a}{x^*-a^*}\\
&= \sup_{(a,a^*)\in \gr A}
\big(\scal{x}{a^*}+\scal{a}{x^*}-\scal{a}{a^*}\big);
\end{align}
\end{subequations}
see \cite{SF} for the original paper and also 
\cite{Simons} for various extensions, applications, and further references. 
It is known (see \cite{SF}) that 
\begin{equation}
\label{e:220505a}
F_A(x,x^*)\geq \scal{x}{x^*},
\quad
\text{with equality if and only if\;\;}
x^*\in Ax.
\end{equation}

The next result will be useful later. 
\begin{proposition}
\label{p:pre}
Given $x,y,x^*,y^*$ in $X$, we have 
\begin{equation}
F_A(x,y^*)+F_A(y,x^*)-\scal{x}{x^*}-\scal{y}{y^*} 
\geq \scal{y-x}{x^*-y^*},
\end{equation}
and equality holds if and only if $y^*\in Ax$ and $x^*\in Ay$. 
\end{proposition}
\begin{proof}
By \cref{e:220505a}, we have 
\begin{equation}
\label{e:220505b}
F_A(x,y^*) \geq \scal{x}{y^*}
\;\;\text{and}\;\;
F_A(y,x^*) \geq \scal{y}{x^*};
\end{equation}
moreover, equality holds for both inequalities if and only if 
$y^*\in Ax$ and $x^*\in Ay$. 
Adding the two inequalities in \cref{e:220505b}, followed by 
subtracting $\scal{x}{x^*}+\scal{y}{y^*}$ from both sides, gives
\begin{subequations}
\begin{align}
F_A(x,y^*)+F_A(y,x^*)-\scal{x}{x^*}-\scal{y}{y^*}
&\geq \scal{x}{y^*} + \scal{y}{x^*} - \scal{x}{x^*}-\scal{y}{y^*}\\
&= \scal{y-x}{x^*-y^*},
\end{align}
\end{subequations}
as claimed. 
\end{proof}

\begin{fact} {\bf (Fitzpatrick)}
\label{f:Fsub}
We have 
$f\oplus f^* \geq F_{\partial f}$. 
\end{fact}
\begin{proof}
This is contained in \cite{SF}; see also the discussion in 
\cite[Section~2]{BMS}. For completeness, we include the short proof here. 
Let $(a,a^*)\in\gr \partial f$. 
Then 
\begin{subequations}
\begin{align}
\scal{a}{x^*}+\scal{x}{a^*}-\scal{a}{a^*}
&= \big(\scal{a}{x^*}-f(a)\big) +\big(\scal{x}{a^*}-f^*(a^*)\big)\\
&\leq f^*(x^*) + f(x). 
\end{align}
\end{subequations}
%We have equality here iff $x^*\in \partial f(a)$ and $a^*\in\partial f(x)$. 
The result follows by taking the supremum over $(a,a^*)\in\gr \partial f$. 
\end{proof}

\begin{corollary} {\bf (Carlier)}
\label{p:1}
Given $x,y,x^*,y^*$ in $X$, we have 
\begin{equation}
\label{C(1.1)}
G_f(x,x^*) + G_f(y,y^*) \geq \scal{y-x}{x^*-y^*}
\end{equation}
and 
\begin{equation}
\label{C(1.1)=}
G(x,x^*) + G(y,y^*) = \scal{y-x}{x^*-y^*}
\;\;\Leftrightarrow\;\;
\big[y^*\in\partial f(x) \text{\;and\;} x^*\in\partial f(y) \big]. 
\end{equation}
\end{corollary}
\begin{proof}
(See also \cite[Section~1]{Carlier}.)
Using \cref{f:Fsub} and \cref{p:pre}, we always have 
\begin{subequations}
\label{e:ncold}
\begin{align}
G(x,x^*) + G(y,y^*)
&= \big(f(x)+f^*(x^*)-\scal{x}{x^*}\big) + \big(f(y)+f^*(y^*)-\scal{y}{y^*}\big)\\
&= \big(f(x)+f^*(y^*)\big) + \big(f(y)+f^*(x^*)\big)-\scal{x}{x^*} -
\scal{y}{y^*}\\
&\geq
F_{\partial f}(x,y^*) + F_{\partial f}(y,x^*)
-\scal{x}{x^*} - \scal{y}{y^*}\label{e:ncold1}\\
&\geq \scal{y-x}{x^*-y^*}, \label{e:ncold2}
\end{align}
\end{subequations}
which is \cref{C(1.1)}. 
We now turn to the proof of \cref{C(1.1)=}.\\
``$\Rightarrow$'': In this case, we have equality in \cref{e:ncold2}. 
In turn,
the equality characterization in \cref{p:pre} yields
$y^*\in\partial f(x)$ and $x^*\in\partial f(y)$. \\
``$\Leftarrow$'': 
In this case, 
$f(x)+f^*(y^*)=\scal{x}{y^*}$ and $f(y)+f^*(x^*)=\scal{y}{x^*}$.
It follows that 
\begin{subequations}
\begin{align}
\big(f(x)+f^*(y^*)\big) + \big(f(y)+f^*(x^*)\big)-\scal{x}{x^*} -
\scal{y}{y^*}
&= \scal{x}{y^*}+\scal{y}{x^*}-\scal{x}{x^*} - \scal{y}{y^*}\\
&= \scal{y-x}{x^*-y^*}. 
\end{align}
\end{subequations}
Hence the chain of inequalities in \cref{e:ncold} 
is actually a chain of equalities and we are done.
\end{proof}

\section{Minty parametrization}

\label{sec:Minty}

In this section, we employ the Minty parametrization and derive 
results that will be useful later. 
Recalling the standing assumption \cref{e:A}, and given 
points $x,x^*$ in $X$ and $\gamma>0$, we have the
well known equivalences (see, e.g., \cite[Chapter~23]{BC2017}): 
\begin{equation}
\label{e:basic}
x^*\in Ax
\;\;\Leftrightarrow\;\;
x+\gamma x^* \in x+\gamma Ax
\;\;\Leftrightarrow\;\;
x = J_{\gamma A}(x+\gamma x^*),
\end{equation}
where $J_{\gamma A}=(\Id+\gamma A)^{-1}$ is the resolvent of $\gamma A$. 

\begin{lemma}
\label{l:Minty}
Let $x,x^*$ be in $X$, and let $\gamma>0$. 
Set 
\begin{equation}
\label{M1}
a := J_{\gamma A}(x+\gamma x^*)
\;\;\text{and}\;\;
a^* := \frac{x+\gamma x^*-J_{\gamma A}(x+\gamma x^*)}{\gamma}.
\end{equation}
Then 
\begin{equation}
\label{M3}
x+\gamma x^* = a+\gamma a^*, 
\end{equation}
\begin{equation}
\label{M2}
(a,a^*)\in\gr A,
\end{equation}
and 
\begin{equation}
\label{e:key1}
\scal{a-x}{x^*-a^*} = \frac{\|x-J_{\gamma A}(x+\gamma x^*)\|^2}{\gamma}. 
\end{equation}
Moreover, 
\begin{equation}
\label{M6}
\|(x,x^*)-(a,a^*)\|^2 = \big(1+\tfrac{1}{\gamma^{2}}\big)\|x-J_{\gamma A}(x+\gamma x^*)\|^2 
\end{equation}
and 
\begin{equation}
\label{M8}
\big\|(x-a)-(x^*-a^*)\big\|^2 = \big( 1+ \tfrac{1}{\gamma}\big)^2\|x-J_{\gamma A}(x+\gamma x^*)\|^2. 
\end{equation}
Finally, we have the equivalences (which may or may not hold)
\begin{equation}
\label{e:key1=0}
x^*\in Ax
\;\;\Leftrightarrow\;\;
a=x
\;\;\Leftrightarrow\;\;
a^*=x^*
\;\;\Leftrightarrow\;\;
\scal{a-x}{x^*-a^*} = 0. 
\end{equation}
\end{lemma}
\begin{proof}
Clearly, \cref{M1} implies \cref{M3}. 
By applying the classical Minty parametrization to $\gamma A$, 
we have $(a,\gamma a^*)\in\gr(\gamma A)$ and thus $(a,a^*)\in \gr A$, 
i.e., \cref{M2} holds. 
Note that 
\begin{equation}
\label{e:220426a}
x-a = x-J_{\gamma A}(x+\gamma x^*)
\end{equation}
and 
\begin{subequations}
\label{e:220426b}
\begin{align}
x^*-a^* 
&= \frac{\gamma x^*-\gamma a^*}{\gamma}
= \frac{\gamma x^* - \big(x+\gamma x^*-J_{\gamma A}(x+\gamma x^*\big)}{\gamma}\\
&= -\frac{1}{\gamma}\big(x - J_{\gamma A}(x+\gamma x^*)\big). 
\end{align}
\end{subequations}
Combining \cref{e:220426a} and \cref{e:220426b}, we deduce 
\begin{align}
\label{e:220505c}
-\scal{x-a}{x^*-a^*} 
&= 
\frac{\|x-J_{\gamma A}(x+\gamma x^*)\|^2}{\gamma},
\end{align}
which is \cref{e:key1}. 
Next, using \cref{e:220426a} and \cref{e:220426b} again, we obtain 
\begin{subequations}
\label{e:220505d}
\begin{align}
\|(x,x^*)-(a,a^*)\|^2
&= 
\|x-a\|^2 + \|x^*-a^*\|^2\\
&= 
\|x-J_{\gamma A}(x+\gamma x^*)\|^2 
+ \frac{\|x-J_{\gamma A}(x+\gamma x^*)\|^2}{\gamma^2}, 
\end{align}
\end{subequations}
which yields \cref{M6}. 
Moreover, using \cref{e:220426a}, \cref{e:220426b}, and 
\cref{e:220505c}, we deduce that 
\begin{subequations}
\begin{align}
\big\|(x-a)-(x^*-a^*)\big\|^2
&= \|x-a\|^2 + \|x^*-a^*\|^2 - 2\scal{x-a}{x^*-a^*}\\
&= \big(1+\tfrac{1}{\gamma^2}+\tfrac{2}{\gamma} \big)\|x-J_{\gamma A}(x+\gamma x^*)\|^2, 
\end{align}
\end{subequations}
which yields \cref{M8}.
We now turn to \cref{e:key1=0}. 
We rewrite \cref{e:220505d} as 
$\|x-a\|^2+\|x^*-a^*\|^2 
= \big(1+\tfrac{1}{\gamma^2}\big)\|x-a\|^2$.
Hence, invoking also \cref{e:220505c}, we obtain
\begin{equation}
\gamma\|x^*-a^*\|^2 = 
\tfrac{1}{\gamma}\|x-a\|^2 = -\scal{x-a}{x^*-a^*},
\end{equation}
which yields the equivalences 
%\cref{e:key1=0}. 
$x=a$
$\Leftrightarrow$
$x^*=a^*$
$\Leftrightarrow$
$\scal{x-a}{x^*-a^*}=0$.
% $\Leftrightarrow$
% $(x,x^*)=(a,a^*)$.
If $x=a$, then $(x,x^*)=(a,a^*)\in \gr A$ by \cref{M2}.
And if $(x,x^*)\in\gr A$, 
then \cref{e:basic} yields $x=a$.
All this proves \cref{e:key1=0}.
\end{proof}

We now record a duality result.

\begin{lemma}
\label{l:dually}
Let $x,x^*$ be in $X$, and let $\gamma>0$.
Then 
\begin{equation}
\label{e:220509a}
x^*-J_{\gamma^{-1} A^{-1}}(x^*+\gamma^{-1} x)
= -\gamma^{-1}\big(x-J_{\gamma A}(x+\gamma x^*)\big);
\end{equation}
consequently,
\begin{equation}
\label{e:220509b}
\frac{\|x^*-J_{\gamma^{-1} A^{-1}}(x^*+\gamma^{-1} x)\|^2}{\gamma^{-1}}
=\frac{\|x-J_{\gamma A}(x+\gamma x^*)\|^2}{\gamma}. 
\end{equation}
\end{lemma}
\begin{proof}
Using \cite[Proposition~23.20]{BC2017}, we have
\begin{subequations}
\begin{align}
x^*-J_{\gamma^{-1} A^{-1}}(x^*+\gamma^{-1} x)
&= x^* - \big(\Id-\gamma^{-1}J_{\gamma A}\circ \gamma\Id\big)(x^*+\gamma^{-1} x)\\
&= x^* - (x^*+\gamma^{-1} x)+\gamma^{-1}J_{\gamma A}(x+\gamma x^*)\\
&= -\gamma^{-1}\big(x-J_{\gamma A}(x+\gamma x^*)\big),
\end{align}
\end{subequations}
which is \cref{e:220509a} and from which \cref{e:220509b} follows.
\end{proof}

\section{The Carlier bound and duality}

\label{sec:Carlier}

This section contains a review of Carlier's inequality and a new duality result. 

\begin{definition} {\bf (Carlier bound)}
Recall \cref{e:A}, let $x$ and $x^*$ be in $X$, and let $\gamma>0$.
We define the associated \emph{Carlier bound} by 
\begin{equation}
\label{e:Carlier}
C(x,x^*) := C_{A,\gamma}(x,x^*) := \frac{\|x-J_{\gamma A}(x+\gamma x^*)\|^2}{\gamma}. 
\end{equation}
If $A=\partial f$, then 
\begin{equation}
C(x,x^*) = C_{\partial f,\gamma}(x,x^*) = \frac{\|x-\prox_{\gamma f}(x+\gamma x^*)\|^2}{\gamma}. 
\end{equation}
\end{definition}

\begin{theorem} {\bf (Carlier)}
\label{t:CarlierA}
Recall \cref{e:A}, and let $x,x^*$ be in $X$. 
Then 
\begin{equation}
(\forall \gamma>0)\quad
F_A(x,x^*)-\scal{x}{x^*} \geq C_{A,\gamma}(x,x^*). 
\end{equation}
\end{theorem}
\begin{proof}
(See also \cite[Section~2]{Carlier}.)
Let $\gamma>0$ and 
set 
\begin{equation}
a := J_{\gamma A}(x+\gamma x^*)
\;\;\text{and}\;\;
a^* := \frac{x+\gamma x^*-J_{\gamma A}(x+\gamma x^*)}{\gamma}.
\end{equation}
By \cref{F_A} and \cref{e:key1}, we have 
\begin{subequations}
\begin{align}
F_A(x,x^*) -  
\scal{x}{x^*} &= -\inf_{(b,b^*)\in\gr A}\scal{x-b}{x^*-b^*}\\
&\geq \scal{a-x}{x^*-a^*}\\
&=\frac{\|x-J_{\gamma A}(x+\gamma x^*)\|^2}{\gamma}\\
&=C_{A,\gamma}(x,x^*), 
\end{align}
\end{subequations}
as claimed. 
\end{proof}

\begin{theorem} {\bf (Carlier)}
\label{t:Carlier}
Let $x$ and $x^*$ be in $X$, and let $\gamma>0$.
Then 
\begin{equation}
\label{C(1.2)}
G_f(x,x^*) \geq F_{\partial f}(x,x^*)-\scal{x}{x^*} 
\geq C_{\partial f,\gamma}(x,x^*). 
\end{equation}
Moreover, we have the characterization
\begin{subequations}
\label{R1.2}
\begin{align}
&\hspace{-1 cm} G_f(x,x^*)=C_{\partial f,\gamma}(x,x^*)
\;\;\Leftrightarrow\;\;\\
&\big[\prox_{\gamma f}(x+\gamma x^*)\in \partial f^*(x^*)
\;\;\text{and}\;\;
x+\gamma x^* - \prox_{\gamma f}(x+\gamma x^*) \in \gamma \partial f(x)\big].
\end{align}
\end{subequations}
% In turn is equivalent to the existence of a vector $q\in X$ such that 
% $q\in Ax$ and $\{x,q\}\subseteq A(x+\gamma(x^*-q))$; if this is the case, then 
% $q$ is necessarily equal to $(x+\gamma x^*-J_{\gamma A}(x+\gamma x^*))/\gamma$. 
\end{theorem}
\begin{proof}
(See also \cite[Section~1]{Carlier}.)
Write $A=\partial f$ so that $A^{-1} = \partial f^*$. 
Now set 
\begin{equation}
a := J_{\gamma A}(x+\gamma x^*)
\;\;\text{and}\;\;
a^* := \frac{x+\gamma x^*-J_{\gamma A}(x+\gamma x^*)}{\gamma}.
\end{equation}
Then \cref{C(1.2)} follows from \cref{f:Fsub} and \cref{t:CarlierA}. 
We now derive this differently in order to characterize equality. 
By \cref{M2}, we have $(a,a^*)\in\gr A$. 
Hence, \cref{e:FY=} yields
\begin{equation}
G_f(a,a^*)=0. 
\end{equation}
Applying now \cref{C(1.1)} and \cref{e:key1} yields
\begin{subequations}
\label{e:220505e}
\begin{align}
G_f(x,x^*)&=G_f(x,x^*)+G_f(a,a^*)\\
&\geq 
\scal{a-x}{x^*-a^*}\\
&=\frac{\|x-J_{\gamma A}(x+\gamma x^*)\|^2}{\gamma}.
\end{align}
\end{subequations}
Moreover, thanks to \cref{C(1.1)=}, we have equality characterization 
\begin{equation}
G_f(x,x^*)=\frac{\|x-J_{\gamma A}(x+\gamma x^*)\|^2}{\gamma}
\;\;\Leftrightarrow\;\;
\big[a^*\in Ax\;\text{and}\;x^*\in Aa], 
\end{equation}
which is precisely \cref{R1.2}.
\end{proof}

\begin{remark}
In view of \cref{t:CarlierA}, 
the Carlier bound
\begin{equation}
C_{A,\gamma}(x,x^*) = \frac{\|x-J_{\gamma A}(x+\gamma x^*)\|^2}{\gamma} 
\end{equation}
is always less sharp than the Fitzpatrick function bound
\begin{equation}
F_{A}(x,x^*)-\scal{x}{x^*};
\end{equation}
however, Carlier's lower bound is sharper than the trivial lower bound 
$0$. 
Computing Fitzpatrick functions is not an easy task 
(see \cite{BMS}) --- there are many 
more examples of prox operators available (see \cite{BC2017} and 
\cite{Beck2}). 
See also \cite{BDL1} and \cite{BDL2} for recent work on Fitzpatrick functions and related objects. 
\end{remark}

Let us observe a new duality result, which links the Carlier bound of $A$
to that of $A^{-1}$: 

\begin{theorem} {\bf (duality)} 
\label{t:duality}
Recall \cref{e:A}, let $x,x^*$ be in $X$, and let $\gamma>0$. 
Then 
\begin{equation}
C_{A,\gamma}(x,x^*) = C_{A^{-1},\gamma^{-1}}(x^*,x). 
\end{equation}
\end{theorem}
\begin{proof}
Combine \cref{e:Carlier} with \cref{e:220509b}.
\end{proof}

We conclude this section by outlining another possible area where
Carlier's inequality may be useful -- Bregman distances!

\begin{remark} {\bf (Bregman distance)} 
Recall \cref{e:f} and that the Bregman distance between $x\in X$ and $y\in\intdom f$
is defined by 
\begin{equation}
D_f(x,y) = f(x)-f(y)-\scal{x-y}{\nabla f(y)}. 
\end{equation}
Note that 
\begin{equation}
\label{e:220513a}
D_f(x,y) = G_f(x,\nabla f(y))\geq C_{\partial f,\gamma}(x,\nabla f(y))
\end{equation}
by \cref{t:Carlier}. 
The Bregman distance plays a role, e.g., when analyzing the proximal gradient method (PGM).
If $(y_n)_\nnn$ is the sequence generated by the PGM and $x$ is a solution, then 
$\sum_\nnn D_f(x,y_n)<\infty$ (see, e.g., the proof of \cite[Theorem~10.21]{Beck2}). It follows 
that $D_f(x,y_n)\to 0$ and also we learn from \cref{e:220513a} that 
\begin{equation}
\sum_{\nnn} C_{\partial f,\gamma}(x,\nabla f(y_n))= 
\sum_\nnn \frac{\|x-\prox_{\gamma f}(x+\gamma \nabla f(y_n))\|^2}{\gamma} <\pinf.
\end{equation}
This is a prototypical appearance of Carlier's inequality in the context of the 
analysis of algorithms. 
\end{remark}

\section{Asymptotic behaviour}

\label{sec:asymp}

Let us now analyze the behaviour of Carlier's bound
\begin{equation}
C_{A,\gamma}(x,x^*) = \frac{\|x-J_{\gamma A}(x+\gamma x^*)\|^2}{\gamma}
\end{equation}
when $\gamma\to 0^+$. Because of \cref{t:duality}, we also obtain information about the behaviour 
when $\gamma\to\pinf$.

%\medskip

\begin{theorem}
\label{t:asy}
Recall \cref{e:A}, let $x,x^*$ be in $X$, and let $\gamma>0$. 
Then the following hold: 
\begin{enumerate}
\item 
\label{t:asy1}
If $x\notin\cdom A$, then $\lim_{\gamma\to 0^+} C_{A,\gamma}(x,x^*)
= \pinf$.
\item 
\label{t:asy2}
If $x\in\dom A$, then $\lim_{\gamma\to 0^+} C_{A,\gamma}(x,x^*)=0$. 
\end{enumerate}
\end{theorem}
\begin{proof}
\cref{t:asy1}: Suppose that $x\notin\cdom A$. 
Set $\delta := d_{\cdom A}(x)>0$. 
Because $\ran J_{\gamma A} = \dom(\gamma A) = \dom A$, we estimate 
\begin{equation}
\frac{\|x-J_{\gamma A}(x+\gamma x^*)\|^2}{\gamma}
\geq \frac{\delta^2}{\gamma}.
\end{equation}
This yields the conclusion. 

\cref{t:asy2}: Suppose that $x\in\dom A$. 
Recall that $\moyo{A}{\gamma}x = (x-J_{\gamma A}x)/\gamma$ by definition of the Yosida approximation. 
Because resolvents are nonexpansive, we have 
\begin{equation}
\label{e:220509c}
\|J_{\gamma A}x-J_{\gamma A}(x+\gamma x^*)\|\leq \gamma\|x^*\|. 
\end{equation}
Clearly, 
\begin{subequations}
\begin{align}
\|x-J_{\gamma A}(x+\gamma x^*)\|^2
&=
\|x-J_{\gamma A}x\|^2 + \|J_{\gamma A}x-J_{\gamma A}(x+\gamma x^*)\|^2\\
&\qquad +2\scal{x-J_{\gamma A}x}{J_{\gamma A}x-J_{\gamma A}(x+\gamma x^*)} 
\end{align}
\end{subequations}
and this implies 
\begin{subequations}
\label{e:220503a}
\begin{align}
\frac{\|x-J_{\gamma A}(x+\gamma x^*)\|^2}{\gamma}
&= \gamma\|\moyo{A}{\gamma}x\|^2
+\frac{\|J_{\gamma A}x-J_{\gamma A}(x+\gamma x^*)\|^2}{\gamma}\\
&\qquad + 2\scal{\moyo{A}{\gamma}x}{J_{\gamma A}x-J_{\gamma A}(x+\gamma x^*)}. 
\end{align}
\end{subequations}
Because $x\in\dom A$, we learn from \cite[Corollary~23.46(i)]{BC2017} that
\begin{equation}
\label{e:220509d}
\lim_{\gamma\to 0^+} \moyo{A}{\gamma}x = \moyo{A}{0}x= P_{Ax}(0). 
\end{equation}
Consider the three summands on the right side of \cref{e:220503a}. 
It suffices to show that each one of them goes to $0$ as $\gamma\to 0^+$. 
First, 
$\moyo{A}{\gamma}x\to P_{Ax}(0)$ and thus 
$\gamma\|\moyo{A}{\gamma}x\|^2\to 0\|P_{Ax}(0)\|^2 = 0$.
Second, \cref{e:220509c} yields 
$0\leq (1/\gamma)\|J_{\gamma A}x-J_{\gamma A}(x+\gamma x^*)\|^2
\leq \gamma\|x^*\|^2 \to 0$ and therefore
$(1/\gamma)\|J_{\gamma A}x-J_{\gamma A}(x+\gamma x^*)\|^2\to 0$. 
Thirdly, \cref{e:220509c} shows that 
$J_{\gamma A}x-J_{\gamma A}(x+\gamma x^*)\to 0$. 
Combined with \cref{e:220509d}, we deduce that 
$2\scal{\moyo{A}{\gamma}x}{J_{\gamma A}x-J_{\gamma A}(x+\gamma x^*)}\to 0$. 
\end{proof}

\begin{corollary}
\label{c:asy}
Recall \cref{e:A}, let $x,x^*$ be in $X$, and let $\gamma>0$. 
Then the following hold: 
\begin{enumerate}
\item 
\label{c:asy1}
If $x^*\notin\cran A$, then $\lim_{\gamma\to \pinf} C_{A,\gamma}(x,x^*)
= \pinf$.
\item 
\label{c:asy2}
If $x^*\in\ran A$, then $\lim_{\gamma\to \pinf} C_{A,\gamma}(x,x^*)=0$. 
\end{enumerate}
\end{corollary}
\begin{proof}
\cref{t:duality} yields
\begin{equation}
C_{A,\gamma}(x,x^*) = C_{A^{-1},\gamma^{-1}}(x^*,x). 
\end{equation}
The result is now clear from \cref{t:asy} (applied to $A^{-1}$)
because $\ran A = \dom A^{-1}$. 
\end{proof}

\begin{corollary}
We have 
\begin{equation}
\dom \sup_{\gamma>0} C_{A,\gamma} \subseteq \cdom A \times \cran A.
\end{equation}
\end{corollary}
\begin{proof}
Combine \cref{t:asy} with \cref{c:asy}. 
\end{proof}

\section{Examples}
\label{sec:examples}

In this section, we collect several examples to illustrate our results.

\begin{example} {\bf (indicator of a subspace)} 
Suppose that $A=N_U$, where $U$ is a closed linear subspace of $X$.
By \cite[Example~3.1]{BMS}, the Fitzpatrick bound at $(x,x^*)\in X\times X$ is 
\begin{equation}
F_{N_U}(x,x^*) -\scal{x}{x^*}= \iota_U(x) + \iota_{U^\perp}(x^*)-\scal{x}{x^*} 
= \big(\iota_U \oplus \iota_{U^\perp}\big)(x,x^*). 
\end{equation}
Let $\gamma>0$. Because $J_{\gamma A}=P_U$ is linear, we compute 
Carlier's bound via
\begin{subequations}
\begin{align}
C_{N_U,\gamma}(x,x^*)
&= \frac{\|x-P_{U}(x+\gamma x^*)\|^2}{\gamma}
= \frac{\|P_{U^\perp}x\|^2+\gamma^2\|P_Ux^*\|^2}{\gamma}\\
&= \frac{1}{\gamma}\|P_{U^\perp}x\|^2 + \gamma\|P_Ux^*\|^2.
\end{align}
\end{subequations}
Note that $X\times X\times \RPP\to\RR\colon (x,x^*,\gamma)\mapsto C_{N_U,\gamma}(x,x^*)$ is 
\emph{not} convex; however, 
$X\times X\to\RR\colon (x,x^*)\mapsto C_{N_U,\gamma}(x,x^*)$ 
and 
$\RPP\to\RR\colon \gamma\mapsto C_{N_U,\gamma}(x,x^*)$ are convex. 
Let us discuss further 
\begin{equation}
\RPP\to\RR\colon \gamma\mapsto C_{N_U,\gamma}(x,x^*).
\end{equation}
This function is 
(i) strictly increasing if $x\in U$ and $x^*\notin U^\perp$;
(ii) strictly decreasing if $x\notin U$ and $x^*\in U^\perp$;
(iii) first strictly decreasing then strictly increasing if
$x\notin U$ and $x^*\notin U^\perp$; 
(iv) identically equal to $0$ if $x\in U$ and $x^*\in U^\perp$. 
Moreover,
\begin{equation}
\lim_{\gamma\to 0^+} C_{N_U,\gamma}(x,x^*) = \iota_U(x)
\;\;\text{and}\;\;
\lim_{\gamma\to \pinf} C_{N_U,\gamma}(x,x^*) = \iota_{U^\perp}(x^*). 
\end{equation}
It follows that 
\begin{equation}
\sup_{\gamma>0} C_{N_U,\gamma}(x,x^*)=\iota_U(x)+\iota_{U^\perp}(x^*)
\end{equation}
coincides with the Fitzpatrick bound in this case. 
\end{example}

\begin{example} {\bf (energy)} 
\label{ex:earlyId}
Suppose that $f=\thalb\|\cdot\|^2$ and hence $\nabla f = \Id$. 
By \cite[Example~3.10]{BMS}, the Fitzpatrick bound at $(x,x^*)\in X\times X$ is 
\begin{equation}
F_{\Id}(x,x^*)-\scal{x}{x^*} = \tfrac{1}{4}\|x+x^*\|^2 - \scal{x}{x^*}=
\tfrac{1}{4}\|x-x^*\|^2. 
\end{equation}
Let $\gamma>0$. Then $\prox_{\gamma f} = (\Id+\gamma\Id)^{-1}= (1+\gamma)^{-1}\Id$ and hence Carlier's bound is 
\begin{subequations}
\label{e:CearlyId}
\begin{align}
C_{\Id,\gamma}(x,x^*) 
&= \frac{\|x-\prox_{\gamma f}(x+\gamma x^*)\|^2}{\gamma}
= \frac{\|x-(1+\gamma)^{-1}(x+\gamma x^*)\|^2}{\gamma}\\
&= \frac{\gamma}{(1+\gamma)^2}\|x-x^*\|^2. 
\end{align}
\end{subequations}
If $x\neq x^*$, then 
$\gamma\mapsto C_{\Id,\gamma}(x,x^*)$ is strictly concave
on $\left]0,2\right]$, strictly convex on $\left[2,+\infty\right[$, 
and its unique global maximizer is $\gamma=1$ for which 
$C_{\Id,1}(x,x^*)=\tfrac{1}{4}\|x-x^*\|^2 = F_{\Id}(x,x^*)$. 
\end{example}

\begin{example} {\bf (skew rotator)} 
Suppose that $X=\RR^2$ and that 
$A\colon\RR^2\to\RR^2\colon(x_1,x_2)\mapsto(-x_2,x_1)$, the counter-clockwise
rotator by $\pi/2$, which is a skew isometry. 
By \cite[Proposition~7.4]{BBW07}, 
\begin{equation}
F_A(x,x^*) = \iota_{\gr A}(x,x^*)=F_A(x,x^*)-\scal{x}{x^*}.
\end{equation}
Let $\gamma>0$. Then $J_{\gamma A} = (1+\gamma^2)^{-1}(\Id-\gamma A)\colon 
(x_1,x_2)\mapsto (1+\gamma^2)^{-1}(x_1+\gamma x_2,-\gamma x_1 + x_2)$. 
Thus, after some algebra, we find that Carlier's bound is 
\begin{align}
C_{A,\gamma}(x,x^*) 
&=
\frac{\|x-J_{\gamma A}(x+\gamma x^*)\|^2}{\gamma}\\
&=
\frac{\gamma}{1+\gamma^2}\|Ax-x^*\|^2.
\end{align}
Therefore, if $Ax\neq x^*$, then 
$\gamma\mapsto C_{A,\gamma}(x,x^*)$ is strictly concave
on $\big]0,\sqrt{3}\big]$, strictly convex on 
$\textstyle \big[\sqrt{3},+\infty\big[$, 
and its unique global maximizer is $\gamma=1$ for which 
$C_{A,1}(x,x^*)=\tfrac{1}{2}\|Ax-x^*\|^2$. 
\end{example}

We now focus on the case when $X=\RR$
and thus $A=\partial f$. 
For ease of notation, we will use $(x,y)$ instead of $(x,x^*)$
as we do elsewhere. 

\begin{example} {\bf ((negative) Burg entropy)} 
Suppose that $f(x) = -\ln(x)$ when $x>0$, and $\pinf$ elsewhere, 
and let $\gamma>0$. 
It is known 
(see, e.g., \cite[Example~6.9]{Beck2} and \cite[Example~24.40]{BC2017}) 
that for $z\in\RR$, 
\begin{equation}
\prox_{\gamma f}(z) = \frac{z+\sqrt{z^2+4\gamma}}{2}. 
\end{equation}
Hence, for $(x,y)\in\RR^2$, 
\begin{subequations}
\begin{align}
C_{\partial f,\gamma}(x,y)
&= 
\frac{|x-\prox_{\gamma f}(x+\gamma y)|^2}{\gamma}\\
&= \frac{\big(x-\tfrac{1}{2}(x+\gamma y)-\tfrac{1}{2}\sqrt{(x+\gamma y)^2+4\gamma}\big)^2}{\gamma}\\
&= 
\frac{\big((x-\gamma y)-\sqrt{(x+\gamma y)^2+4\gamma}\big)^2}{4\gamma};
\end{align}
\end{subequations}
in particular, 
\begin{subequations}
\label{e:220517b}
\begin{align}
C_{\partial f,\gamma}(0,y)
&= 
\frac{\big((0-\gamma y)-\sqrt{(0+\gamma y)^2+4\gamma}\big)^2}{4\gamma}\\
&= 
\frac{\big(\gamma y+\sqrt{(\gamma y)^2+4\gamma}\big)^2}{4\gamma}\\
&= 
\frac{\big(\sqrt{\gamma}\sqrt{\gamma}y+\sqrt{\gamma}\sqrt{\gamma y^2+4}\big)^2}{(2\sqrt{\gamma})^2}\\
&=\bigg(\frac{\sqrt{\gamma}y + \sqrt{\gamma y^2+4}}{2}\bigg)^2\\
&\to 1 \qquad\text{as $\gamma\to 0^+$.}
\end{align}
\end{subequations}
Combining with \cref{t:asy}, we obtain
\begin{equation}
\lim_{\gamma\to 0^+} C_{\partial f,\gamma}(x,y) = 
\begin{cases}
+\infty, &\text{if $x<0$;}\\
1, &\text{if $x=0$;}\\
0, &\text{if $x>0$}
\end{cases}
\end{equation}
which is convex --- but not lower semicontinuous --- as a function of $x$. 
%\hl{To do:} Compare to \cite[Example~3.4]{BMS}. 
\end{example}

\begin{example} {\bf ((negative) Boltzmann-Shannon entropy)} 
Suppose that $f$ at $x\in\RR$ is defined by 
\begin{equation}
f(x)=\begin{cases}
+\infty,&\text{if }x<0;\\
0,&\text{if }x=0;\\
x\ln(x)-x,&\text{if }x>0.
\end{cases}
\end{equation}
We start by showing that 
\begin{equation}
\label{e:220517a}
\prox_{\gamma f} (x)=\gamma  W\parens*{\frac{1}{\gamma} 
\exp\Big(\frac{x}{\gamma}\Big)}% \text{ for },
\end{equation}
where $x\in\RR$ and 
where $W$ is the \emph{Lambert W-function} as defined in \cite[Equation 1.5]{W}.
To see that, recall that (see, e.g., \cite[Proposition~24.1]{BC2017}) the 
characterization of the proximal mapping
\begin{equation}
p=\prox_{\gamma f} (x) \siff \gamma\nabla f(p) + p=x\siff
\gamma \ln (p)+p=x,
\end{equation}
where $x\in\RR$ and $p>0$. 
Hence
\begin{equation}
\frac{p}{\gamma}\exp\Big(\frac{p}{\gamma}\Big)=\frac{1}{\gamma}\exp\Big(\frac{x}{\gamma}\Big) \siff p=\gamma W\parens*{\frac{1}{\gamma}\exp\Big(\frac{x}{\gamma}\Big)}
\end{equation}
by the very definition of the $W$ function, and this verifies \cref{e:220517a}. 
Therefore
\begin{subequations}
\begin{align}
C_{\partial f,\gamma}(x,y)
&= 
\frac{|x-\prox_{\gamma f}(x+\gamma y)|^2}{\gamma}\\
&= \frac{\parens*{x-\gamma W\parens*{\frac{1}{\gamma}
\exp\big((x+\gamma y)/\gamma\big)}}^2}{\gamma}.
\end{align}
\end{subequations}
In particular, 
\begin{subequations}
\begin{align}
C_{\partial f,\gamma}(0,y) 
&= \gamma\parens*{ W\parens*{\tfrac{1}{\gamma}\exp( y)}}^2\\
&=\frac{\parens*{ W\parens*{\frac{1}{\gamma}\exp( y)}}^2}{\tfrac{1}{\gamma}}. 
\end{align}
\end{subequations}
We wish to take now the limit as $\gamma\to 0^+$.
As numerator and denominator tend to $\pinf$, we shall use 
L'Hospital's rule. Using the fact that 
\begin{equation}
W'(z) = \frac{1}{(1 + W(z)) \exp(W(z))},
\end{equation}
we obtain
 \begin{equation}
 \lim_{\gamma\to 0^+} C_{\partial f,\gamma}(0,y)
 = 
 \lim_{\gamma\to 0^+}\frac{2 W\big(\frac{1}{\gamma}\exp( y)\big)\exp( y)}{\exp\Big(W\big(\frac{1}{\gamma}\exp( y)\big)\Big) \parens*{1+W\big(\frac{1}{\gamma}\exp( y)\big)}}.
\end{equation}
Changing variables via
$u=W\big(\frac{1}{\gamma}\exp( y)\big)$,
we finally obtain 
\begin{equation}
\label{e:220517c}
 \lim_{\gamma\to 0^+} C_{\partial f,\gamma}(0,y)
 = 
\lim_{u\to+\infty}\frac{2u\exp(y)}{\exp(u)(1+u)}=0.
\end{equation}
Combining this with \cref{t:asy} gives us
\begin{equation}
\lim_{\gamma\to 0^+} C_{\partial f,\gamma}(x,y) = 
\begin{cases}
+\infty, &\text{if $x<0$;}\\
0, &\text{if $x\geq 0$}.
\end{cases}
\end{equation}
\end{example}

\begin{remark}
The formulas \cref{e:220517b} and \cref{e:220517c} illustrate
that the asymptotic behaviour at boundary points does not seem
to follow a simple pattern and thus warrants further study.
\end{remark}

\section{Cyclic monotonicity}

\label{sec:ncyc}

In this section, we extend the analysis to $n$-cyclically monotone 
operators. 
Recall that $A$ is $n$-cyclically monotone, where $n\in\{2,3,\ldots\}$, if 
\begin{equation}
\left. 
\begin{array}{c}
(a_1,a_1^*)\in\gr A\\
(a_2,a_2^*)\in\gr A\\
\;\;\vdots\\
(a_n,a_n^*)\in\gr A\\
a_{n+1}=a_1
\end{array}
\right\}
\;\;\Rightarrow\;\;
\sum_{k=1}^{n} \scal{a_{k+1}-a_k}{a_k^*}\leq 0. 
\end{equation}
$2$-cyclic monotonicity is just regular monotonicity.
The \emph{Fitzpatrick function of order $n$}, $F_{A,n}$, 
evaluated at $(x,x^*)\in X\times X$, is the supremum over 
$(a_1,a_1^*),\ldots,(a_{n-1},a_{n-1}^*)$ in $\gr A$ of the expression
\begin{equation}
\label{e:FAn}
\scal{x}{x^*} + \scal{x-a_{n-1}}{a_{n-1}^*} + \scal{a_1-x}{x^*} + 
\sum_{k=1}^{n-2}\scal{a_{k+1}-a_k}{a_k^*}.
\end{equation}
As a supremum of continuous affine functions, 
the function $F_{A,n}$ is convex and lower semicontinuous. 
We also set $F_{A,\infty} = \sup_{n\geq 2} F_{A,n}$. 

\begin{fact} (See \cite[Corollary~2.8]{BBBRW}.)
Suppose that $A$ is maximally $n$-cyclically monotone. 
Then $F_{A,n}>\scal{\cdot}{\cdot}$ outside $\gr A$, while $F_{A,n}=\scal{\cdot}{\cdot}$ on $\gr A$. 
\end{fact}

We have (see \cite[Remark~2.10]{BBBRW}) the ordering 
\begin{equation}
\label{e:220510d}
\scal{\cdot}{\cdot} \leq F_{A,2}\leq F_{A,3}\leq \cdots \leq F_{A,n}\to F_{A,\infty}.
\end{equation}
Moreover, if $f$ is as in \cref{e:f}, then 
\cite[Theorem~3.5]{BBBRW} yields for every $(x,x^*)\in X\times X$ 
\begin{equation}
\label{e:220510e}
F_{\partial f,\infty}(x,x^*) = f(x)+f^*(x^*). 
\end{equation}
Computing $F_{A,n}$ is nontrivial; 
for some concrete examples, see 
\cite[Section~4]{BBBRW} and also \cite{BBW07}. 

We shall need the following identity.

\begin{lemma}
Let $(x,x^*)\in X$, and 
let $(a_1,a_1^*),\ldots,(a_{n-1},a_{n-1}^*)$ be in $X\times X$.
Then
\begin{subequations}
\label{e:ncyclem}
\begin{align}
\scal{x-a_{n-1}}{a_{n-1}^*} + \scal{a_1-x}{x^*}+\sum_{k=1}^{n-2}
\scal{a_{k+1}-a_k}{a_k^*}\\ = 
\scal{a_1-x}{x^*-a_1^*} + \sum_{k=2}^{n-1}\scal{a_k-x}{a_{k-1}^*-a_k^*}. 
\end{align}
\end{subequations}
\end{lemma}
\begin{proof}
We prove this by induction on $n\geq 2$. 
If $n=2$, then the left side of \cref{e:ncyclem} is 
\begin{equation}
\scal{x-a_{1}}{a_{1}^*} + \scal{a_1-x}{x^*}
= \scal{a_1-x}{x^*-a_1^*},
\end{equation}
which is also equal to the right side of \cref{e:ncyclem}. 

Now assume the result is true for some $n\geq 2$. 
We will show the result is also true for $n+1$. 
Indeed, using the inductive hypothesis in \cref{e:220510a}, we have 
\begin{subequations}
\begin{align}
&\hspace{-1cm}\scal{x-a_{n}}{a_{n}^*} + \scal{a_1-x}{x^*}+\sum_{k=1}^{n-1}
\scal{a_{k+1}-a_k}{a_k^*}\\ 
&= 
\scal{x-a_{n}}{a_{n}^*} -\scal{x-a_{n-1}}{a_{n-1}^*} + 
\scal{a_n-a_{n-1}}{a_{n-1}^*} \\
&\qquad + \scal{x-a_{n-1}}{a_{n-1}^*}+ \scal{a_1-x}{x^*}+\sum_{k=1}^{n-2} \scal{a_{k+1}-a_k}{a_k^*}\\ 
&=\scal{x-a_{n}}{a_{n}^*} -\scal{x-a_{n-1}}{a_{n-1}^*}
+\scal{a_n-a_{n-1}}{a_{n-1}^*}
\label{e:220510a}\\
&\qquad 
+ \scal{a_1-x}{x^*-a_1^*} + \sum_{k=2}^{n-1}\scal{a_k-x}{a_{k-1}^*-a_k^*} \\
&=\scal{x-a_{n}}{a_{n}^*} -\scal{x-a_{n-1}}{a_{n-1}^*}
+\scal{a_n-a_{n-1}}{a_{n-1}^*}-\scal{a_{n}-x}{a_{n-1}^*-a_n^*}\\
&\qquad 
+ \scal{a_1-x}{x^*-a_1^*} + \sum_{k=2}^{n}\scal{a_k-x}{a_{k-1}^*-a_k^*}
\\
&=\scal{a_1-x}{x^*-a_1^*} + \sum_{k=2}^{n}\scal{a_k-x}{a_{k-1}^*-a_k^*}
\end{align}
\end{subequations}
and we are done. 
\end{proof}

\begin{theorem}
Let $(x,x^*)\in X\times X$. If 
$(a_1,a_1^*),\ldots,(a_{n-1},a_{n-1}^*)$ belong to $\gr A$, then 
\begin{align}
\label{e:220510b}
F_{A,n}(x,x^*)-\scal{x}{x^*}
&\geq 
\scal{a_1-x}{x^*-a_1^*} + \sum_{k=2}^{n-1}\scal{a_k-x}{a_{k-1}^*-a_k^*}.
\end{align}
Moreover, let $\gamma_1>0,\ldots,\gamma_{n-1}>0$
and set 
\begin{subequations}
\begin{align}
a_1 &:= J_{\gamma_1A}(x+\gamma_1x^*), 
&a_1^* &:= \frac{x+\gamma_1x^*-a_1}{\gamma_1}\\
a_2 &:= J_{\gamma_2A}(x+\gamma_2a_1^*),
&a_2^* &:= \frac{x+\gamma_2a_1^*-a_2}{\gamma_2}\\
&\;\;\;\vdots & &\;\;\;\vdots&\\
a_{n-1} &:= J_{\gamma_{n-1}A}(x+\gamma_{n-1}a_{n-2}^*),
&a_{n-1}^* &:= \frac{x+\gamma_{n-1}a_{n-2}^*-a_{n-1}}{\gamma_{n-1}}.
\end{align}
\end{subequations}
Then 
\begin{align}
\label{e:220510c}
F_{A,n}(x,x^*)-\scal{x}{x^*}
&\geq 
\frac{\|x-J_{\gamma_1A}(x+\gamma_1x^*)\|^2}{\gamma_1}+ \sum_{k=2}^{n-1}
\frac{\|x-J_{\gamma_kA}(x+\gamma_ka_{k-1}^*)\|^2}{\gamma_k}.
\end{align}
\end{theorem}
\begin{proof}
Indeed, using \cref{e:FAn} and \cref{e:ncyclem}, we have 
\begin{subequations}
\begin{align}
F_{A,n}(x,x^*)-\scal{x}{x^*}
&\geq 
 \scal{x-a_{n-1}}{a_{n-1}^*} + \scal{a_1-x}{x^*} + 
\sum_{k=1}^{n-2}\scal{a_{k+1}-a_k}{a_k^*}\\
&=\scal{a_1-x}{x^*-a_1^*} + \sum_{k=2}^{n-1}\scal{a_k-x}{a_{k-1}^*-a_k^*}
\end{align}
\end{subequations}
which is \cref{e:220510b}.

We now turn towards the ``Moreover'' part.
By \cref{l:Minty}, 
the pairs $(a_1,a_1^*),\ldots,(a_{n-1},a_{n-1}^*)$ lie in $\gr A$. 
The inequality \cref{e:220510c} follows by combining 
\cref{e:220510b} with \cref{e:key1}.
\end{proof}

\begin{corollary} {\bf (a series lower bound)} 
\label{c:superC}
Recall that $f$ satisfies \cref{e:f}, and let $x,x^*$ be in $X$. 
Let $(\gamma_n)_{n\geq 1}$ be a sequence in $\RPP$.
Generate $(a_n)_{n\geq 1}$
and $(a_n^*)_{n\geq 0}$ via 
\begin{equation}
\label{e:superseq}
a_0^* := x^*,\qquad 
(\forall n\geq 1)\quad 
a_n := \prox_{\gamma_n f}(x+\gamma_na_{n-1}^*)
\;\;\text{and}\;\;
a_n^* := \frac{x+\gamma_n a_{n-1}^*-a_n }{\gamma_n}. 
\end{equation}
Then we obtain the lower bound
%\begin{subequations}
%\begin{align}
\begin{equation}
\label{e:superC}
G_f(x,x^*)=f(x)+f^*(x^*)-\scal{x}{x^*}
\geq 
\sum_{k=1}^\infty 
\frac{\|x-\prox_{\gamma_kf}(x+\gamma_ka_{k-1}^*)\|^2}{\gamma_k}.
%\\
% &\geq 
% \frac{\|x-\prox_{\gamma_1f}(x+\gamma_1x^*)\|^2}{\gamma_1}. 
\end{equation}
%\end{align}
%\end{subequations}
\end{corollary}
\begin{proof}
Because $\partial f$ is maximally cyclically monotone, 
the result thus follows by combining 
\cref{e:220510d}, \cref{e:220510e}, and \cref{e:220510c}. 
\end{proof}

\begin{remark}
Consider \cref{c:superC} and its notation. 
Suppose that $(x,x^*)\in\dom f\times \dom f^*$.
Then \cref{e:superC} yields 
\begin{equation}
\sum_{k=1}^\infty 
\frac{\|x-\prox_{\gamma_kf}(x+\gamma_ka_{k-1}^*)\|^2}{\gamma_k} < \pinf;
\quad\text{hence,}\quad
\frac{x-\prox_{\gamma_kf}(x+\gamma_ka_{k-1}^*)}{\sqrt{\gamma_k}} \to 0. 
\end{equation}
If we truncate the infinite series in \cref{e:superC} after the first term and 
set $\gamma = \gamma_1$,
then we obtain Carlier's bound  (see \cref{t:Carlier})
\begin{equation}
f(x)+f^*(x^*)-\scal{x}{x^*} 
\geq \frac{\|x-\prox_{\gamma f}(x+\gamma x^*)\|^2}{\gamma}. 
\end{equation}
\end{remark}

\begin{example}
Suppose that $f=\iota_U$, where $U$ is a closed linear subspace of $X$.
Then $\prox_{\gamma_nf} = P_U$ and 
for every $n\geq 1$, we have 
\begin{equation}
a_n = P_U(x+\gamma_n a_{n-1}^*) = P_Ux + \gamma_n P_Ua_{n-1}^* 
\end{equation}
and 
\begin{subequations}
\begin{align}
a_n^* 
&= \frac{x+\gamma_n a_{n-1}^*-a_n }{\gamma_n} 
= \frac{x+\gamma_na_{n-1}^*-P_Ux-\gamma_nP_Ua_{n-1}^*}{\gamma_n}\\
&= \frac{P_{U^\perp}x}{\gamma_n}+P_{U^\perp}a_{n-1}^*\in U^\perp\\
&\;\; \vdots\\
&=\bigg(\frac{1}{\gamma_1}+\frac{1}{\gamma_2}+\cdots+\frac{1}{\gamma_n}\bigg)P_{U^\perp}x + P_{U^\perp}x^*\in U^\perp;
\end{align}
\end{subequations}
thus, 
$a_1 = P_Ux+\gamma_1P_Ux^*$
and
$a_2 = \cdots = a_n = P_Ux$.
It follows that 
\begin{equation}
\frac{\|x-\prox_{\gamma_1f}(x+\gamma_1a_{0}^*)\|^2}{\gamma_1}
= \frac{\|x-P_{U}(x+\gamma_1x^*)\|^2}{\gamma_1}
= \frac{1}{\gamma_1}\|P_{U^\perp}x\|^2 + \gamma_1\|P_{U}x^*\|^2
\end{equation}
and that for every $k\geq 2$
\begin{subequations}
\begin{align}
\frac{\|x-\prox_{\gamma_kf}(x+\gamma_ka_{k-1}^*)\|^2}{\gamma_k}
&=\frac{\|x-P_{U}(x+\gamma_ka_{k-1}^*)\|^2}{\gamma_k}
= \frac{1}{\gamma_k}\|P_{U^\perp}x\|^2+\gamma_k\|P_Ua_{k-1}^*\|^2\\
&= \frac{1}{\gamma_k}\|P_{U^\perp}x\|^2. 
\end{align}
\end{subequations}
Therefore, the lower bound in \cref{e:superC} turns into 
\begin{equation}
\gamma_1\|P_Ux^*\|^2 + \bigg(\sum_{k=1}^\infty \frac{1}{\gamma_k}\bigg)\|P_{U^\perp}x\|^2
\end{equation}
which is strictly larger than Carlier's bound whenever $x\notin U$. 
\end{example}

\begin{example}
Suppose that $f=\tfrac{1}{2}\|\cdot\|^2 = f^*$, let 
$\gamma>0$ and set $\gamma_n=\gamma$ for all $n\geq 1$. 
Then 
$\prox_{\gamma_kf}= (\Id+\gamma_k\Id)^{-1}=(1+\gamma)^{-1}\Id$. 
Then for every $n\geq 1$, we have
\begin{equation}
a_n =  \frac{x+\gamma a_{n-1}^*}{1+\gamma} 
\end{equation}
and 
\begin{subequations}
\begin{align}
a_n^* 
&= \frac{x+\gamma a_{n-1}^*-a_n }{\gamma} 
= \frac{x+\gamma a_{n-1}^*-(1+\gamma)^{-1}(x+\gamma a_{n-1}^*)}{\gamma}\\
&= \frac{1}{1+\gamma}x + \frac{\gamma}{1+\gamma}a_{n-1}^*\\
&\;\; \vdots\\
&=\bigg(1-\frac{\gamma^n}{(1+\gamma)^n}\bigg)x + \frac{\gamma^n}{(1+\gamma)^n}x^*. 
\end{align}
\end{subequations}
It follows that 
\begin{equation}
\frac{\|x-\prox_{\gamma_1f}(x+\gamma_1a_{0}^*)\|^2}{\gamma_1}
= \frac{\|x-(1+\gamma)^{-1}(x+\gamma x^*)\|^2}{\gamma}
= \frac{\gamma}{(1+\gamma)^2}\|x-x^*\|^2 
\end{equation}
and that for every $k\geq 2$
\begin{subequations}
\begin{align}
\frac{\|x-\prox_{\gamma_kf}(x+\gamma_ka_{k-1}^*)\|^2}{\gamma_k}
&=\frac{\gamma}{(1+\gamma)^2}\|x-a_{k-1}^*\|^2\\
&=\frac{\gamma}{(1+\gamma)^2}\frac{\gamma^{2(k-1)}}{(1+\gamma)^{2(k-1)}}\|x-x^*\|^2\\
&=\frac{\gamma^{2k-1}}{(1+\gamma)^{2k}}\|x-x^*\|^2. 
\end{align}
\end{subequations}
That is, 
\begin{equation}
(\forall k\geq 1)\quad
\frac{\|x-\prox_{\gamma_kf}(x+\gamma_ka_{k-1}^*)\|^2}{\gamma_k} = \frac{\gamma^{2k-1}}{(1+\gamma)^{2k}}\|x-x^*\|^2. 
\end{equation}
Therefore, the lower bound in \cref{e:superC} turns into 
\begin{equation}
\bigg(\sum_{k=1}^\infty \frac{\gamma^{2k-1}}{(1+\gamma)^{2k}} \bigg)
\|x-x^*\|^2 
= \frac{\gamma}{1+2\gamma}\|x-x^*\|^2
\end{equation}
which is strictly greater than Carlier's bound $\gamma(1+\gamma)^{-2}\|x-x^*\|^2$ 
whenever $x\neq x^*$. 
\end{example}

\section*{Acknowledgments}
We thank Guillaume Carlier for sending us his beautiful preprint \cite{Carlier}. 
HHB and XW were supported by NSERC Discovery Grants.

\end{document}